\documentclass{amsart}
\usepackage{latexsym}
\usepackage{amsmath,amsfonts,epsfig, graphics, graphicx, amsthm, amssymb}
\input{xy}
\xyoption{all}

\newtheorem{theorem}{Theorem}[section]
\newtheorem{lemma}[theorem]{Lemma}

\newtheorem*{GRH}{Generalized Riemann Hypothesis}

\pdfoutput=1

\begin{document}
\title{Primes in Arithmetic Progressions to Large Moduli and Siegel Zeroes}
\author[T. Wright]{Thomas Wright}
\address{Wofford College\\429 N. Church St.\\Spartanburg, SC 29302\\USA}
\maketitle

\begin{abstract}
Let $\chi$ be a Dirichlet character mod $D$ with $L(s,\chi)$ its associated $L$-function, and let $\psi(x,q,a)$ be Chebyshev's prime-counting function for primes congruent to $a$ modulo $q$.  We show that under the assumption of an exceptional character $\chi$ with $L(1,\chi)=o\left((\log D)^{-5}\right)$, for any $q<x^{\frac 23-\varepsilon}$, the asymptotic
$$\psi(x,q,a)=\frac{\psi(x)}{\phi(q)}\left(1-\chi\left(\frac{aD}{(D,q)}\right)+o(1)\right)$$
holds for almost all $a$ with $(a,q)=1$.  We also find that for any fixed $a$, the above holds for almost all $q<x^{\frac 23-\varepsilon}$ with $(a,q)=1$.
Previous prime equidistribution results under the assumption of Siegel zeroes (by Friedlander-Iwaniec and the current author) have found that the above asymptotic holds either for all $a$ and $q$ or on average over a range of $q$ (i.e. for the Elliott-Halberstam conjecture), but only under the assumption that $q<x^{\theta}$ where $\theta=\frac{30}{59}$ or $\frac{16}{31}$, respectively.
\end{abstract}

\section{Introduction}

We recall first the definition of the Dirichlet $L$-function:
$$L(s,\chi)=\sum_{n=1}^\infty \frac{\chi(n)}{n^s}.$$
Here, $\chi$ is a Dirichlet character modulo an integer $q>2$.  We will assume that $\chi$ is non-principal, and hence the above sum is convergent for $Re(s)>0$.

The study of primes in arithmetic progressions is closely related to the study of when $L(s,\chi)$ equals zero.  Dirichlet's work found a zero-free region around $s=1$, while larger zero-free regions would allow for better error terms for this theorem.  Indeed, one of the most famous conjectures in mathematics is the belief that all of these zeroes are, in fact, on the half-line:

\begin{GRH}
For a Dirichlet character $\chi$, let $L(s,\chi)=0$ for $s=\sigma+it$ with $\sigma>0$.  Then $\sigma=\frac 12$.
\end{GRH}

Of course, we are nowhere close to proving this.  In the case where the zero is real, the best effective and ineffective bounds come from Landau \cite{La} and Siegel \cite{Si}, respectively:
\begin{theorem}[Landau, 1918]
There exists an effectively computable positive constant $C$ such that for any $q$ and any character $\chi$ mod $q$, if $L(s,\chi)=0$ and $s$ is real, then
$$s<1-\frac{C}{q^\frac 12 \log^2 q}.$$
\end{theorem}
\begin{theorem}[Siegel, 1935]
For any $\varepsilon>0$ there exists a positive constant $C(\varepsilon)$ such that if $L(s,\chi)=0$ and $s$ is real then
$$s<1-C(\varepsilon)q^{-\varepsilon}.$$
\end{theorem}
However, most zeroes are far closer to the half-line than these bounds indicate.  In fact, it is known (see \cite{Gr}, \cite{La}, \cite{Ti}) that for any $q$, every zero of $L(s,\chi)$ except at most one will obey a much smaller bound:
\begin{theorem}
There is an effectively computable positive constant $C$ such that $$\prod_{\chi\mbox{ }mod\mbox{ }q}L(s,\chi)=0$$
has at most one solution on the region $$\sigma\geq 1-\frac{C}{\log q(2+|t|)}.$$
If such a zero exists, $s$ must be real, and the character for which $L(s,\chi)=0$ must be a non-principal real character.
\end{theorem}


A zero of this type, if it is to exist, is called a \textit{Siegel zero} or an \textit{exceptional zero}, and the associated character is called an exceptional character. We note that the definition given here (or, indeed, in the literature in general) for a Siegel zero is not particularly rigorous, since this definition depends on the choice of the constant $C$.

\section{Siegel Zeroes}\label{Siegelsection}

Interestingly, the existence of Siegel zeroes would lead to some surprisingly nice properties among the primes.  Most notably, the existence of Siegel zeroes would allow us to prove (among other things) the twin prime conjecture \cite{HB83}, small gaps between general $m$-tuples of primes \cite{WrS}, the existence of large intervals where the Goldbach conjecture is true \cite{MaMe}, a hybrid Chowla and Hardy-Littlewood conjecture \cite{TT}, and results about primes in arithmetic progressions that would allow the modulus $q$ to be greater than $\sqrt x$ \cite{FI03}, \cite{WrSiAP}.  It is this last result that is of interest in the present paper.

In the definitions below, we will assume that $(a,q)=1$.  We recall that Chebyshev's functions are given by
\begin{gather*}
\psi(x)=\sum_{n\leq x}\Lambda(n),\\
\psi(x,q,a)=\sum_{\substack{n\leq x\\n\equiv a\pmod q}}\Lambda(n),
\end{gather*}
where $\Lambda$ is the von Mangoldt function given by
$$\Lambda(n)=\begin{cases} \log p & \mbox{if }n=p^k\mbox{ for prime }p,\\ 0 & \mbox{otherwise.}\end{cases}$$
In 2003, Friedlander and Iwaniec \cite{FI03} proved the following:

\begin{theorem}[Friedlander-Iwaniec, 2003] \label{FIFI}
Let $\chi$ be a real character mod $D$.  Let $x>D^r$ with $r=554,401$, let $q=x^\theta$ with $\theta<\frac{233}{462}$, and let $(a,q)=1$.  Then
\begin{gather}\label{FIthm}\psi(x,q,a)=\frac{\psi(x)}{\phi(q)}\left(1-\chi\left(\frac{aD}{(D,q)}\right)+O\left(L(1,\chi)(\log x)^{r^r}\right)\right).
\end{gather}
\end{theorem}


While this theorem gives us an understanding of the distribution of primes in arithmetic progressions beyond the so-called $x^\frac 12$ barrier, it requires a rather extreme Siegel zero.  Indeed, this theorem is only non-trivial if
$$L(1,\chi)=o\left(\left(\log D\right)^{-{554,401}^{554,401}}\right).$$
In a recent work \cite{WrSiAP}, the current author relaxed the requirements on $\theta$ and $L(1,\chi)$ to $\theta<\frac{30}{59}-\varepsilon$ and
$$L(1,\chi)=o\left(\left(\log D\right)^{-7}\right),$$
in addition to proving that (\ref{FIthm}) holds for almost all $q\sim x^\theta$ with $\theta<\frac{16}{31}-\varepsilon$.

It was noted in \cite{FI03} that if the methods of that paper have an obvious technical obstruction at $\theta=\frac 23$, which we will discuss below.  In light of this, it seems that $\theta=\frac 23-\varepsilon$ might be the best possible result using these methods.  In this paper, we show that in several contexts, we can indeed reach this bound for $\theta$.

Since we generally only require $L(1,\chi)$ to be smaller than log to a power, we define the following.  For some large power of $A$ (say, $A=10,000$), write
$$\mathcal L(\chi)=\max\{L(1,\chi),\log^{-A}x\}.$$
Throughout this paper, we will generally assume that any $\varepsilon<\frac{1}{500}$ and $\alpha<\frac{1}{500}$, as this will be sufficiently small for our purposes.

We will prove two results.  The first is a fairly sharp version of the Brun-Titchmarsh theorem for $q<x^{\frac 23-\varepsilon}$, which can be restated as a lower bound on the number of primes congruent to $a$ mod $q$ for every $q<x^{\frac 23-\varepsilon}$ and almost all $a$ with $(a,q)=1$:

\begin{theorem}\label{Main1}
Let $x$ and $D$ be such that $\log D=(\log x)^{\kappa}$ for some $\kappa<1$, let $\sqrt x<q<D^{-1}x^{\frac 23-\varepsilon}$ for any $\varepsilon>0$, and let $(a,q)=1$.
Then
$$\psi(x,q,a)\leq \frac{1-\chi_D\left(\frac{aD}{(D,q)}\right)+O\left(\mathcal L(\chi)\log^5 x\right)}{\phi(q)}\psi(x).$$
Moreover, for a given $q$ as above, for any function $h=h(x)<1$, the equation
$$\psi(x,q,a)\geq \frac{1-h-\chi_D\left(\frac{aD}{(D,q)}\right)}{\phi(q)}\psi(x),$$
holds for all but $$O\left(\frac{\phi(q)}{h}\left(\mathcal L(\chi)\log^5 x\right)\right)$$values of $a$ with $(a,q)=1$.
\end{theorem}
This result is obviously non-trivial if $L(1,\chi)=o\left(\log^{-5} x\right)$.

If $D|q$ and $\chi(a)=1$, the above bound can be improved.
\begin{theorem}\label{Main1.5}
Let $x$, $D$, $q$, and $a$ be as in Theorem \ref{Main1}, with the additional restriction that $D|q$ and $\chi(a)=1$.   Then
$$\psi(x,q,a)\leq \frac{xL(1,\chi)\log x}{q}+O\left(Dq^{\frac 12+\alpha}\right).$$
\end{theorem}
This is non-trivial if $L(1,\chi)\leq \frac{c}{\log x}$ for some $C$.

For the other result, we recall that the Bombieri-Vinogradov theorem and the Elliott-Halberstam conjecture consider the question of which $\theta$ allow for the following inequality to hold for arbitrary $A$:
\begin{gather}\label{EHfull}\sum_{q\leq x^\theta}\max_{(a,q)=1}\left|\psi(x,q,a)-\frac{\psi(x)}{\psi(a)}\right|\ll \frac{x}{\log^A x}.
\end{gather}
It is known that (\ref{EHfull}) holds for $\theta<1/2$, and the Elliott-Halberstam conjecture posits that this holds for all $\theta<1-\varepsilon$ for any $\varepsilon>0$.  Under the assumption of Siegel zeroes, it has been proven by the current author \cite{WrSiAP} that this holds for $\theta=\frac{16}{31}-\varepsilon$.

However, if one changes the inequality in (\ref{EHfull}) to consider the sum over a fixed congruence class, we can consider instead the weaker question of which $\theta$ allow for the following:
\begin{gather}\label{EHfixed}\sum_{\substack{q\sim x^\theta \\ (a,q)=1}}\left|\psi(x,q,a)-\frac{\psi(x)}{\phi(q)}\right|=o\left(x\right).
\end{gather}
It is known that when considering (\ref{EHfixed}) instead of (\ref{EHfull}), one can move slightly past $\theta=1/2$ to $\theta=1/2+h(x)$ for any function $h$ such that $h(x)=o(1)$ \cite{BFI}. (One can also move even further past $\theta$ for both (\ref{EHfull}) and (\ref{EHfixed}) if one is to consider specific well-chosen subsets of $q\sim Q$ - see e.g. \cite{MaWF1} for a more thorough discussion of such results.)

Since (\ref{EHfixed}) appears to be slightly more tractable than (\ref{EHfull}) in the unconditional setting, it would stand to reason that this would be true under the assumption of a Siegel zero as well.  Here, we prove that this is indeed the case:

\begin{theorem}\label{Main2}
Let $a\in \mathbb Z$, let $x$ and $D$ be such that $\log D=(\log x)^{\kappa}$ for some $\kappa<1$, and let $Q$ be such that $Q<x^{\frac 23-\varepsilon}$.  Then
$$\sum_{\substack{q\sim Q \\ (a,q)=1}}\left|\psi(x,q,a)-\frac{\psi(x)}{\phi(q)}\right|\ll x\mathcal L(\chi)\log^5 x.$$
\end{theorem}

\section{$\lambda$ in arithmetic progressions}\label{Deflambda}



Let $\ast$ denote the Dirichlet convolution, let $\chi=\chi_D$ be an exceptional character mod $D$, and define
\begin{gather*}
\lambda=\chi\ast 1,\\
\lambda'=\chi\ast \log,\\
\nu=\mu\ast\mu\chi,
\end{gather*}
as well as the following sums:
\begin{gather*}
S'(x,q,a)=\sum_{\substack{n\leq x \\n\equiv a \pmod q}}\lambda'(n),\\
S(x,q,a)=\sum_{\substack{n\leq x \\n\equiv a \pmod q}}\lambda(n),\\
S'(x,q)=\sum_{\substack{n\leq x \\(n,q)=1}}\lambda'(n),\\
S(x,q)=\sum_{\substack{n\leq x \\(n,q)=1}}\lambda(n).
\end{gather*}
In \cite{FI03}, the authors evaluate $\psi(x,q,a)$ by exploiting the fact that
$$\Lambda=\mu\ast \log =\mu\ast \log \ast (1\ast\mu)\chi=\log\ast \chi\ast \mu\ast\chi\mu=\lambda'\ast \nu.$$
The authors first show that
$$\mathop{\sum\sum}\limits_{\substack{dm\leq x \\ dm\equiv a\pmod q \\ d\leq D^2}}\lambda'(m)\nu(d)=\frac{1-\chi\left(\frac{aD}{(D,q)}\right)}{\phi(q)} \mathop{\sum\sum}\limits_{\substack{dm\leq x \\ dm\equiv a\pmod q \\ d\leq D^2}}\lambda'(m)\nu(d)+O\left(q^{\frac 12+\varepsilon}+\frac{x}{q}\log^3 xL(1,\chi)\right).$$
The rest of the paper then shows that
\begin{gather}\label{FIerror}\mathop{\sum\sum}\limits_{\substack{dm\leq x \\ dm\equiv a\pmod q \\ d>D^2}}\lambda'(m)\nu(d)\ll \frac xq\log^A xL(1,\chi)+q^{\frac{115}{58}+\alpha}\end{gather}
for a large value of $A$ and a small $\alpha$.  To do this, they use the fact that $|\nu(d)|\leq \lambda(d)$ and then apply a method of Landreau \cite{Land} to reduce the quaternary sum $\lambda'\ast \lambda$ to a ternary sum $\lambda\ast 1$, which allows them to apply results from ternary sums in arithmetic progressions.  The work of \cite{WrSiAP} treats these ternary sums a bit more carefully and then also uses more recent Kloosterman sum results to improve the error term in (\ref{FIerror}).

In this paper, we will use the exceptional character $\chi$ to mimic the M\"{o}bius function $\mu$ more directly.  More specifically, an exceptional character $\chi$ is a multiplicative function for which $\chi(p)=-1$ for ``most" primes $p$ in a large region beginning around $D$ and ending at some large $D_L$, where $D_L$ depends on both $D$ and the location of the zero.  So for a natural number $n$ whose divisors are on this interval, it will usually be the case that $\chi(n)=\mu(n)$.  So let
$$R=\max\{D^5,xe^{-(\log x)^{\frac 12}}\}.$$
For a multiplicative function like $\mu$, Tao and Teravainen \cite{TT} coined the term "Siegel model" to describe the function $\mu_{Siegel}$ where $\mu_{Siegel}=\mu(p)$ if $p\leq R$ and $\mu_{Siegel}=\chi(p)$ if $p>R$.  We simplify this idea to consider only $R$-rough numbers, as we define
$$\lambda_R'(n)=\lambda'(n)\textbf{1}_{P(n)>R},$$
$$\lambda_R(n)=\lambda(n)\textbf{1}_{P(n)>R},$$
where $P(n)$ denotes the smallest prime divisor of $n$.
For ease of computation, we also define the square-free analogue to $\lambda_R'$ and $\lambda'_R$:
$$\lambda'_W(n)=\lambda_R'(n)\mu(n)^2,$$
$$\lambda_W(n)=\lambda_R(n)\mu(n)^2,$$
which will be helpful since $\lambda(jk)=\lambda(j)\lambda(k)$ if $jk$ is square-free.

Since $\chi(n)=\mu(n)$ should be true for most $n$ in the support of $\lambda_W'$, it should usually be the case that $\lambda_W'(n)=\Lambda(n)$.  We can then use our understanding of $\lambda'$ in arithmetic progressions mod $q$ to gain a similar understanding of $\lambda'_W$ in arithmetic progressions.  Since $\Lambda(n)\leq \lambda'(n)$ when $n$ is not a prime power, we can then use this to give an upper bound for Chebyshev's function $\psi(x,q,a)$:
\begin{gather}\label{psiSW}\psi(x,q,a)\leq \sum_{\substack{n\leq x \\ n\equiv a\pmod q}}\lambda'_W(n)=\frac{x(1-\chi_D\left(\frac{aD}{(D,q)}\right)+o(1))}{\phi(q)}.\end{gather}
As this upper bound is actually fairly sharp, we know that for almost all $a$, the inequality on the left-hand side can actually be replaced with an equality, else one would end up with $\psi(x)$ having a main term smaller than $x$, which is impossible.

For the second main theorem, a key insight is that for any nonnegative function $f$,
\begin{gather}\label{fg}
\sum_{q\sim Q}\mathop{\sum}\limits_{\substack{n\leq x \\ n\equiv a\pmod q}}f(n)\leq \mathop{\sum}\limits_{\substack{n\leq x }}f(n)\tau(n-a),
\end{gather}
where $\tau(n)$ is the divisor function.  Landreau \cite{Land} introduced the idea of bounding a divisor function via inequalities of the form
$$\tau(n)\ll \sum_{\substack{d|n \\ d<n^\frac 1r}}\tau(d)^{\beta}$$
for some $\beta$ that depends on $r$, and Friedlander and Iwaniec \cite{FI03} later applied this to results about Siegel zeroes.  We will use the strongest available bound of this form, which comes from Munshi \cite{Mu}:
\begin{lemma}[Munshi, 2011]\label{Mun}
For any natural number $n$ and any real number $r>2$,
$$\tau(n)\ll \sum_{\substack{d|n \\ d\leq n^\frac 1r}}\tau(d)^\beta,$$
where
$$\beta=-\frac{\log r}{\log 2}+r\left(1+\left(1-\frac 1r\right)\frac{\log\left(1-\frac 1r\right)}{\log 2}\right).$$
In particular, if $r\leq 4$ then $\beta<1$.
\end{lemma}

Hence, we can bound (\ref{fg}) as
\begin{gather}
\ll \sum_{d\leq x^\frac 14}\tau(d)\mathop{\sum}\limits_{\substack{n\leq x \\ n\equiv a\pmod d}}f(n).
\end{gather}
We are now left to consider the inner sum over a smaller modulus than $q$, which makes it easier to show that the sum is small.  In practice, we will be able to decompose $\Lambda$ into a main term, which we can evaluate directly, and an error term, which is a sum that looks like (\ref{fg}).

The paper, then, will proceed in six parts.  Sections \ref{background1} and \ref{background2} establish the behavior of $\lambda'$ and $\lambda$, both modulo $q$ and in general.  Section \ref{lambdaR} shows the functions $\lambda'$ and $\lambda'_R$ are similarly equidistributed in arithmetic progressions, while Section \ref{lambdaW} shows that this distribution holds for $\lambda'_W$ as well.  Section \ref{final} uses the relationship between $\lambda'_W$ and $\Lambda$ to find an upper bound for $\psi(x,q,a)$ and prove Theorem \ref{Main1}.  Finally, Section \ref{final2} uses the trick in (\ref{fg}) to prove Theorem \ref{Main2}

We noted earlier that there is a technical obstruction at $\theta=2/3$.  The obstruction comes from the fact that these methods work by analyzing the behavior of a twisted $k$-fold divisor function in arithmetic progressions.  In \cite{FI03}, Friedlander and Iwaniec removed one variable from the quaternary divisor function $\lambda\ast \lambda'$ so that they could consider instead a ternary divisor function $\lambda\ast 1$.  Since the ternary divisor function allows one to find information modulo $q$ for a $q>\sqrt x$, this allowed the authors to move $q$ past the $x^\frac 12$ barrier to $\theta=\frac{233}{462}$.  Our insight is that one can further reduce this problem to one about a twisted binary divisor function, and since the binary divisor function is understood modulo $q$ for $q\leq x^{\frac 23-\varepsilon}$, our result allows for $q$ up to this bound.  It is not clear that one could move beyond the $x^{\frac 23}$-barrier, however, while still using the machinery of divisor functions.

\section{Background Lemmas: Distribution of $\lambda$}\label{background1}

In this section, we examine the distribution of $\lambda$ and $\lambda'$ in arithmetic progressions.  For the latter, we simply restate Proposition 4.2 of \cite{FI03}:

\begin{lemma}\label{lambda'ap} For any $\alpha>0$,
$$S'(x,q,a)=\frac{1-\chi_D\left(\frac{aD}{(D,q)}\right)}{\phi(q)}S'(x,q)+O\left((Dq)^{\frac 12+\alpha}\right).$$
This is non-trivial if $Dq<x^{\frac 23-3\alpha}$.
\end{lemma}
Following largely the same framework as the \cite{FI03} proof of the above, we can prove a similar theorem about $\lambda$:

\begin{lemma}\label{lambdaap} For any $\alpha>0$,
$$S(x,q,a)=\frac{1+\chi_D\left(\frac{aD}{(D,q)}\right)}{\phi(q)}S(x,q)+O\left((Dq)^{\frac 12+\alpha}\right).$$
This is non-trivial if $Dq<x^{\frac 23-3\alpha}$.
\end{lemma}
\begin{proof}
Let $f(u)$ be a smooth function supported on $(0,x+y)$ such that $f(u)=1$ on $[1,x]$.  In particular, we will set $y=q^{1+\alpha}$. Moreover, define
$$S_f(x,q,a)=\sum_{n\equiv a \pmod q}\lambda(n)f(n)$$
and
$$S_f(x,q)=\sum_{n\in \mathbb N}\lambda(n)f(n).$$
Note that by Shiu's theorem \cite{Sh},
$$S_f(x,q,a)=S(x,q,a)+O\left(\sum_{\substack{x\leq n\leq x+y \\ n\equiv a \pmod q}}\tau(n)\right)=S(x,q,a)+O\left(q^{2\alpha}\right),$$
and similarly,
$$S_f(x,q)=S(x,q)+O\left(q^{1+2\alpha}\right).$$
For any character $\chi_q$ mod $q$, write
$$Z(s,\chi_q)=\sum_{n\in \mathbb N}\lambda(n)\chi(n)n^{-s}=L(1,\chi_D\chi_q)L(1,\chi_q).$$
and define
$$\tilde{f}(s)=\int_0^\infty f(u)u^{s-1}du.$$
Letting $(1+\varepsilon)$ denote the line at $Re(s)=1+\varepsilon$, we can express $S_f(x,q,a)$ as a contour integral:
$$S_f(x,q,a)=\frac{1}{\phi(q)}\sum_{\chi_q\mbox{ }mod\mbox{ }q}\overline{\chi_q(a)}\frac{1}{2\pi i}\int_{(1+\varepsilon)}Z(s,\chi_q)\tilde{f}(s)ds.$$
Similarly, letting $\chi_0$ denote the principal character mod $q$,
$$S_f(x,q)=\frac{1}{2\pi i}\int_{(1+\varepsilon)}Z(s,\chi_0)\tilde{f}(s)ds,$$
Note that
$$S_f(x,q,a)=\frac{1}{\phi(q)}S_f(x,q)+\frac{1}{\phi(q)}\sum_{\chi_q\neq \chi_0}\overline{\chi_q(a)}\frac{1}{2\pi i}\int_{(1+\varepsilon)}Z(s,\chi_q)\tilde{f}(s)ds.$$
Moreover, if $D|q$ then there exists a character $\chi_q$ such that $\chi_q=\chi_D\chi_0$, and hence we have
$$S_f(x,q,a)=\frac{1+\chi_D(a)}{\phi(q)}S_f(x,q)+\frac{1}{\phi(q)}\sum_{\chi_q\neq \chi_0,\chi_0\chi_D}\overline{\chi_q(a)}\frac{1}{2\pi i}\int_{(1+\varepsilon)}Z(s,\chi_q)\tilde{f}(s)ds.$$
For the remaining sum of characters, we move the contour of integration for each of these integrals to $(-\varepsilon)$.  Since all of the remaining $L(s,\chi)$ have non-principal $\chi$, the L-functions are analytic and hence we have no poles.  So for each such $L$ and $\chi_q$,
\begin{align*}\int_{(1+\varepsilon)}Z(s,\chi_q)\tilde{f}(s)ds= \int_{(-\varepsilon)}L(s,\chi_q)L(s,\chi_q\chi_D)\tilde{f}(s)ds.\end{align*}
We exploit the functional equation
$$L(s,\chi_q)=W(\chi_q)q^{\frac 12-s}X(s)\Gamma(1-s)L(1-s,\chi_q),$$
where  $W(\chi)=\frac{G(\chi)}{\sqrt q}$ is the normalized Gauss sum and $|X(s)|=O(1)$ is dependent only on $s$.  So
\begin{align*}
\frac{1}{\phi(q)}&\sum_{\chi_q\neq \chi_0,\chi_0\chi_D}\int_{(1+\varepsilon)}Z(s,\chi_q)\tilde{f}(s)ds\\
=&\frac{1}{\phi(q)}\sum_{\chi_q\neq \chi_0,\chi_0\chi_D}W(\chi_q)W(\chi_q\chi_D)\overline{\chi_q(a)}\\
&\cdot \frac{1}{2\pi i}\int_{(-\varepsilon)}L(1-s,\chi_q)L(1-s,\chi_q\chi_D)\Gamma(1-s)^2X(s)^2q^{1-2s}D^{\frac 12-s}\tilde{f}(s)ds.
\end{align*}
We can write
$$L(1-s,\chi_q)L(1-s,\chi_q\chi_D)=\sum_{n\in \mathbb N}\frac{c(n)\chi_q(n)}{n^{1-s}},$$
where $|c(n)|\leq \tau(n)$.
Isolating the terms in the integral and sum that have characters $\chi_q$, we then consider
\begin{align*}
\sum_{\chi_q\neq \chi_0,\chi_0\chi_D}&\overline{\chi_q(a)}W(\chi_q)W(\chi_q\chi_D)L(1-s,\chi_q)L(1-s,\chi_D)\\
=&\sum_{n\in \mathbb N}\frac{c(n)}{n^{1-s}}\sum_{\chi_q\neq \chi_0,\chi_0\chi_D}\chi_q(n\bar a)W(\chi_q)W(\chi_q\chi_D).
\end{align*}
By Corollary 4.1 of \cite{FI03} and the fact that $Re(s)=-\varepsilon$, this is
$$\ll \sum_{n\in \mathbb N}\frac{c(n)\sqrt{Dq} }{n^{1+\varepsilon}}\ll \sqrt{Dq}.$$
Moreover, integrating $\tilde{f}$ by parts repeatedly gives
$$\tilde{f}(s)\ll \frac{1}{|s|}\min\left(1,\left(\frac{x}{y|s|}\right)^2\right).$$
Thus,
\begin{align*}
\left|\frac{1}{\phi(q)}\sum_{\chi_q\neq \chi_0,\chi_0\chi_D}\int_{(1+\varepsilon)}Z(s,\chi_q)\tilde{f}(s)ds\right|\ll &\frac{(q)^{\frac 32+2\varepsilon}D^{\frac 12+\varepsilon}}{\phi(q)}\int_{(-\varepsilon)}\tilde{f}(s)ds\ll (Dq)^{\frac 12+3\varepsilon}.
\end{align*}
Letting $\varepsilon=\frac{\alpha}{3}$ then completes the lemma.
\end{proof}

\section{Background Lemmas: Multiplicativity and Decomposition}\label{background2}

Here, we also give two lemmas about the decomposition of $\lambda$ and $\lambda'$ that will help us exploit the multiplicativity of $\lambda$.  The first one allows us to decompose $\lambda'$:

\begin{lemma}\label{split}
Let $(d,n)=1$.  Then
$$\lambda'(dn)=\lambda(d)\lambda'(n)+\lambda'(d)\lambda(n).$$
\end{lemma}

\begin{proof}
We can write
$$\lambda'(dn)=\sum_{l|dn}\chi(l)\log\left(\frac{dn}{l}\right).$$
Since $(d,n)=1$, we can split $l$ uniquely into $l=d_1n_1$ where $d_1|d$ and $n_1|n$.  So
\begin{align*}
\lambda'(dn)=&\sum_{d_1|d}\sum_{n_1|n}\chi(d_1)\chi(n_1)\log\left(\frac{d}{d_1}\cdot \frac{n}{n_1}\right)\\
= &\sum_{d_1|d}\sum_{n_1|n}\left(\chi(d_1)\chi(n_1)\log\left(\frac{d}{d_1}\right) +\chi(d_1)\chi(n_1)\log\left(\frac{n}{n_1}\right)\right)\\
= &\lambda'(d)\lambda(n) +\lambda(d)\lambda'(n).
\end{align*}
\end{proof}

In order to transition from $\lambda_R'$ to $\lambda'_W$ as mentioned in the introduction, we will also need to be able to split $\lambda(n)$ into square-free and square parts.  We give an inequality for this in the following lemma.

\begin{lemma}\label{sqsplit}
Write $n=st$ where $s$ is a square and $t$ is square-free.  Then
$$\lambda(n)\leq \lambda(s)\lambda(t),$$
\end{lemma}
\begin{proof}
We first recall that if $\chi$ is an exceptional character then, for any prime $p$, we have $\chi(p)=\pm 1$ or 0.  If $\chi(p)=-1$ then $\lambda(p^{2k})=1$ and $\lambda(p)=0$, and hence
$$\lambda(p^{2k+1})=0=\lambda(p)\lambda(p^{2k}).$$
If $\chi(p)=0$ then $\lambda(p^{r})=1$ for every $r\geq 1$, and hence
$$\lambda(p^{2k+1})=1=\lambda(p)\lambda(p^{2k}).$$
If $\chi(p)=1$, we have
$$\lambda(p^{2k+1})=\tau(p^{2k+1})=2k+2\leq (2k+1)(2)=\tau(p^{2k})\tau(p)=\lambda
(p^{2k})\lambda(p).$$
So for a given natural number $n$, write
$$n=p_1^{2k_1+j_1}\cdots p_t^{2k_t+j_t}$$
where $k_i$ is a whole number and $j_i$ is 0 or 1. Since $\lambda$ is multiplicative, we then have
\begin{align*}
\lambda(n)=&\lambda(p_1^{2k_1+j_1}\cdots p_t^{2k_t+j_t})=\lambda(p_1^{2k_1+j_1})\cdots \lambda(p_t^{2k_t+j_t})\\
\leq &\lambda(p_1^{2k_1})\lambda(p_1^{j_1})\cdots \lambda(p_t^{2k_t})\lambda(p_t^{j_t})=\lambda(p_1^{2k_1}\cdots p_t^{2k_t})\lambda(p_1^{j_1}\cdots p_t^{j_t})=\lambda(s)\lambda(t).
\end{align*}
\end{proof}
We also note for later that
$$\sum_{n\leq x}\lambda(n)=xL(1,\chi)+O\left(D\sqrt x\right),$$
and
$$\sum_{D^2\leq n\leq x}\frac{\lambda(n)}{n}=L(1,\chi)\log x,$$
as these are Lemma 5.1 and (5.9) of \cite{FI03}, respectively.

\section{A rough $\lambda'$ function}\label{lambdaR}

Recall the definitions of $\lambda'_R$ and $\lambda_R$ given in Section \ref{Deflambda}.  Analogously to the definitions of $S'(x,q)$ and $S'(x,q,a)$, we define
\begin{gather*}
S_R'(x,q)=\sum_{\substack{n\leq x \\(n,q)=1}}\lambda_R'(n),\\
S'_R(x,q,a)=\sum_{\substack{n\leq x \\n\equiv a \pmod q}}\lambda'_R(n).
\end{gather*}
We prove that something like Lemma \ref{lambda'ap} holds for these variants as well.
\begin{lemma}\label{DR}
For any $\alpha>0$,
$$S'_R(x,q,a)=\frac{1-\chi_D\left(\frac{aD}{(D,q)}\right)}{\phi(q)}S_R'(x,q)+O\left((Dq)^{\frac 12+\alpha}+\frac{x\log^5 x\mathcal L(\chi)}{\phi(q)}\right).$$
\end{lemma}
\begin{proof}
Define
$$P(R)=\prod_{p\leq R}p.$$
Then
$$S_R'(x,q,a)=\sum_{\substack{n\leq x \\n\equiv a \pmod q \\ (n,P(R))=1}}\lambda'(n)w(n)=\sum_{\substack{n\leq x \\n\equiv a \pmod q}}\lambda'(n)\sum_{\substack{d|n \\ d|P(R)}}\mu(d)=\mathop{\sum\sum}\limits_{\substack{dm\leq x \\ dm\equiv a\pmod q \\ d|P(R)}}\mu(d)\lambda'(dm).$$
We split the sum over $d$ into ranges.  Choose some small $\eta>0$.  Then
$$S_R'(x,q,a)=\mathop{\sum\sum}\limits_{\substack{dm\leq x,d\leq x^\eta \\ dm\equiv a\pmod q \\ d|P(R)}}\mu(d)\lambda'(dm)+\mathop{\sum\sum}\limits_{\substack{dm\leq x,d>x^\eta \\ dm\equiv a\pmod q \\ d|P(R)}}\mu(d)\lambda'(dm).$$
For the second sum, we note that since $d$ is $R$-smooth, if $d>x^\eta$ then $d$ must have a (not necessarily unique) divisor $l|d$ such that $x^\frac \eta 2\leq l\leq Rx^\frac \eta 2$.  Letting $d=lj$,
\begin{align*}
\mathop{\sum\sum}\limits_{\substack{dm\leq x,d>x^\eta \\ dm\equiv a\pmod q \\ d|P(R)}}&\mu(d)\lambda'(dm)\\
\ll &\log x\sum_{\substack{x^\frac \eta 2\leq l\leq Rx^\frac \eta 2 \\ l|P(R)}}\tau(l)\mathop{\sum\sum}\limits_{\substack{jm\leq \frac xl \\ jm\equiv a\bar l\pmod q }}\tau(jm)\\
\ll & \log x\sum_{l_1\leq x^{\frac \eta 4}}\sum_{\substack{\frac{x^\frac \eta 2}{l_1}\leq l_2\leq \frac{Rx^\frac \eta 2}{l_1} \\ l_2|P(R)}}\mathop{\sum}\limits_{\substack{k\leq \frac x{l_1l_2} \\ k\equiv a\bar l_1\bar l_2\pmod q }}\tau_4(k),
\end{align*}
where $k=jm$ and $l=l_1l_2$.

By Shiu's theorem, the inner sum can be bounded by
$$\mathop{\sum}\limits_{\substack{k\leq \frac x{l_1l_2} \\ k\equiv a\bar l_1\bar l_2\pmod q }}\tau_4(k)\ll \frac{x\log^3 x}{l_1l_2}.$$
For the sum over $l_2$, we can split this sum into dyadic intervals, finding
$$\sum_{\substack{\frac{x^\frac \eta 2}{l_1}\leq l_2\leq \frac{Rx^\frac \eta 2}{l_1} \\ l_2|P(R)}}\frac{1}{l_2}\ll \sum_{r=0}^{\log_2R}\sum_{\substack{2^r\frac{x^\frac \eta 2}{l_1}\leq l_2\leq 2^{r+1}\frac{x^\frac \eta 2}{l_1} \\ l_2|P(R)}}\frac{1}{l_2}\ll \sum_{r=0}^{\log_2R}\frac{l_1}{2^rx^\frac \eta 2}\sum_{\substack{2^r\frac{x^\frac \eta 2}{l_1}\leq l_2\leq 2^{r+1}\frac{x^\frac \eta 2}{l_1} \\ l_2|P(R)}}1. $$
By the standard estimate for smooth numbers \cite{Di}, we know that for
$$u=\frac{\log x}{\log y},$$
the number of $y$-smooth numbers less than $x$ is $\ll x\rho(u)\ll xu^{-u}$, where $\rho$ is the Dickman-de Bruijn function. So this sum over $l_2$ and $r$ is
$$\ll \sum_{r=0}^{\log_2R}e^{-\frac{\eta\log x}{4\log R}}\ll e^{-\frac{\eta\log x}{4\log R}}\log R. $$
So
$$\log x\sum_{l_1\leq x^{\frac \eta 4}}\sum_{\substack{\frac{x^\frac \eta 2}{l_1}\leq l_2\leq \frac{Rx^\frac \eta 2}{l_1} \\ l_2|P(R)}}\mathop{\sum}\limits_{\substack{k\leq \frac x{l_1}{l_2} \\ k\equiv a\bar l_1\bar l_2\pmod q }}\tau_4(k)\ll \frac{x\log^4 x}{q}\sum_{l_1\leq x^{\frac \eta 4}}\frac{1}{l_1}\sum_{\substack{\frac{x^\frac \eta 2}{l_1}\leq l_2\leq \frac{Rx^\frac \eta 2}{l_1} \\ l_2|P(R)}}\frac{1}{l_2}\ll \frac xqe^{-\frac{\eta\log x}{5\log R}}.$$
Hence,
$$S_R'(x,q,a)=\mathop{\sum\sum}\limits_{\substack{dm\leq x,d\leq x^\eta \\ dm\equiv a\pmod q \\ d|P(R)}}\mu(d)\lambda'(dm)+O\left(\frac{x}{q}e^{-\frac{\eta\log x}{5\log R}}\right).$$
For the remaining sum, since $\lambda'$ has the decomposition given in Lemma \ref{split}, we write $m=jk$, where $(k,d)=1$ and $rad(j)|d$.  Then
$$\mathop{\sum\sum}\limits_{\substack{dm\leq x,d\leq x^\eta \\ dm\equiv a\pmod q \\ d|P(R)}}\mu(d)\lambda'(dm)=\sum_{\substack{d\leq x^\eta \\ d|P(R)}}\sum_{\substack{j\leq \frac xd \\ rad(j)|d}}\mathop{\sum}\limits_{\substack{k\leq \frac{x}{dj},\\ dkj\equiv a\pmod q }}[\mu(d)\lambda'(dj)\lambda(k)+\mu(d)\lambda(dj)\lambda'(k)].$$
We can again split the sum over $j$ into $j\leq x^\eta$ and $j>x^\eta$.  Note that for the sum over $j>x^\eta$, we have
$$\log x \sum_{\substack{x^\eta<j\leq \frac xd \\ rad(j)|d|P(R) }}\mathop{\sum}\limits_{\substack{m\leq \frac{x}{dj},\\ dkj\equiv a\pmod q }}\tau(djk)\ll \frac x{dq}e^{-\frac{\eta\log x}{5\log R}}$$
by the same reasoning as before.  So
\begin{align*}\mathop{\sum\sum}\limits_{\substack{dm\leq x,d\leq x^\eta \\ dm\equiv a\pmod q \\ d|P(R)}}&\mu(d)\lambda'(dm)\\
=&\sum_{\substack{d\leq x^\eta \\ d|P(R)}}\sum_{\substack{j\leq x^\eta \\ rad(j)|d}}\mathop{\sum}\limits_{\substack{m\leq \frac{x}{dj},\\ dkj\equiv a\pmod q }}\left[\mu(d)\lambda(dj)\lambda'(k)+\mu(d)\lambda'(dj)\lambda(k)\right]+O\left(\frac x{q}e^{-\frac{\eta\log x}{6\log R}}\right).
\end{align*}
Noting that $\frac{x}{dj}\geq x^{1-2\eta}$ for some small choice of $\eta$, we can apply Lemmas \ref{lambda'ap} and \ref{lambdaap} to find for any $\varepsilon>0$,
\begin{align*}
\sum_{\substack{d\leq x^\eta \\ d|P(R)}}&\sum_{\substack{j\leq x^\eta \\ rad(j)|d}}\mathop{\sum}\limits_{\substack{m\leq \frac{x}{dj},\\ dkj\equiv a\pmod q }}\left[\mu(d)\lambda(dj)\lambda'(k)+\mu(d)\lambda'(dj)\lambda(k)\right]\\
=&\frac{1-\chi_D\left(\frac{aD}{(D,q)}\right)}{\phi(q)}\sum_{\substack{d\leq x^\eta \\ d|P(R)\\ (d,q)=1}}\sum_{\substack{j\leq x^\eta \\ rad(j)|d}}\mathop{\sum}\limits_{\substack{k\leq \frac{x}{dj},\\ (k,q)=1 }}\mu(d)\lambda(dj)\lambda'(k)\\
&+\frac{1+\chi_D\left(\frac{aD}{(D,q)}\right)}{\phi(q)}\sum_{\substack{d\leq x^\eta \\ d|P(R)\\ (d,q)=1}}\sum_{\substack{j\leq x^\eta \\ rad(j)|d}}\mathop{\sum}\limits_{\substack{k\leq \frac{x}{dj},\\ (k,q)=1 }}\mu(d)\lambda'(dj)\lambda(k)+O\left(q^{\frac 12+\varepsilon}x^{2\eta}+\frac xqe^{-\frac{\eta\log x}{6\log R}}\right).
\end{align*}
For the latter expression
\begin{align*}
\frac{1+\chi_D\left(\frac{aD}{(D,q)}\right)}{\phi(q)}\sum_{\substack{d\leq x^\eta \\ d|P(R)\\ (d,q)=1}}\sum_{\substack{j\leq x^\eta \\ rad(j)|d}}\mathop{\sum}\limits_{\substack{k\leq \frac{x}{dj},\\ (k,q)=1 }}\lambda'(dj)\lambda(k)\ll &\frac{x\log xL(1,\chi)}{\phi(q)}\sum_{\substack{d\leq x^\eta}}\sum_{\substack{j\leq x^\eta}}\frac{\tau(dj)}{dj}\\
\ll &\frac{x\log xL(1,\chi)}{\phi(q)}\sum_{\substack{r\leq x^{2\eta}}}\frac{\tau_4(r)}{r}\\
\ll &\frac{x\log^5 xL(1,\chi)}{\phi(q)}.
\end{align*}
So
\begin{align*}
\sum_{\substack{d\leq x^\eta \\ d|P(R)}}&\sum_{\substack{j\leq x^\eta \\ rad(j)|d}}\mathop{\sum}\limits_{\substack{m\leq \frac{x}{dj},\\ dkj\equiv a\pmod q }}\left[\mu(d)\lambda(dj)\lambda'(k)+\mu(d)\lambda'(dj)\lambda(k)\right]\\
=&\frac{1-\chi_D\left(\frac{aD}{(D,q)}\right)}{\phi(q)}\sum_{\substack{d\leq x^\eta \\ d|P(R)\\ (d,q)=1}}\sum_{\substack{j\leq x^\eta \\ rad(j)|d}}\mathop{\sum}\limits_{\substack{k\leq \frac{x}{dj},\\ (k,q)=1 }}\mu(d)\lambda(dj)\lambda'(k)+O\left(q^{\frac 12+\varepsilon}x^{2\eta}+\frac xqe^{-\frac{\eta\log x}{6\log R}}+\frac{x\log^5 xL(1,\chi)}{\phi(q)}\right).
\end{align*}
By essentially the same reasoning,
$$S_R'(x,q)=\sum_{\substack{d\leq x^\eta \\ d|P(R)\\ (d,q)=1}}\sum_{\substack{j\leq x^\eta \\ rad(j)|d}}\mathop{\sum}\limits_{\substack{k\leq \frac{x}{dj},\\ (k,q)=1 }}\mu(d)\lambda'(dj)\lambda'(k)+O\left(xe^{-\frac{\eta\log x}{6\log R}}+x\log^5 xL(1,\chi)\right).$$
Taking $\eta=\frac 14\varepsilon$ and $\varepsilon=\frac{\alpha}{2}$, the lemma then follows.
\end{proof}
\section{A square-free $\lambda'$ function}\label{lambdaW}
Recall now the definitions of $\lambda'_W$ and $\lambda_W$ from Section \ref{Deflambda}.  We next show that the difference between $\lambda'_R$ and $\lambda'_W$ is minimal.
\begin{lemma}\label{DW}
Let $\sqrt x\leq q<D^{-1}x^{\frac 23-3\alpha}$.  Then for any $\alpha>0$,
$$S_W'(x,q,a)=S_R'(x,q,a)+O\left(q^{\frac 12+\alpha}+\frac{x\log x}{qR^{1-2\alpha}}\right),$$
and
$$S_W'(x,q)=S_R'(x,q)+O\left(\frac{x}{R^{1-2\alpha}}\right).$$
Hence,
$$S_W'(x,q,a)=\frac{1-\chi_D\left(\frac{aD}{(D,q)}\right)}{\phi(q)}S_W'(x,q)+O\left((Dq)^{\frac 12+\alpha}+\frac xqe^{-\frac{\varepsilon\log x}{24\log R}}+\frac{x\log^5 xL(1,\chi)}{\phi(q)}\right).$$

\end{lemma}
\begin{proof}
As in Lemma \ref{sqsplit}, we write $n=st$ where $s$ is a square and $t$ is square-free.  Note that $|S_W'(x,q,a)-S_R'(x,q,a)|$ will be a sum comprised of numbers $n$ that have a square factor of size at least $R^2$.  So for any $\alpha>0$,
\begin{align*}|S_W'(x,q,a)-S_R'(x,q,a)|\leq &\sum_{\substack{st\leq x \\ st\equiv a\pmod q \\ s>R^2}}\lambda'_R(st)\\
\ll &\log x\sum_{R^2<s\leq \frac{x}{q^{1+\alpha}}}\tau(s)\sum_{\substack{t\leq \frac xs \\ t\equiv a\bar s\pmod q}}\tau(t)+x^{o(1)}\sum_{\frac{x}{q^{1+\alpha}}<s<q^{1+\alpha}}\sum_{\substack{t\leq \frac xs \\ t\equiv a\bar s\pmod q}}1\\
&+x^{o(1)}\sum_{\substack{R<t<\frac{x}{q^{1+\alpha}}  }}\sum_{\substack{q^{1+\alpha}<s\leq \frac xt \\ \\ s\equiv a\bar t\pmod q}}1
\end{align*}
Write $s=d^2$.  For the first term, we can use Shiu's theorem, along with the fact that $\tau(s)\ll s^{\alpha}$ for any $\alpha>0$:
\begin{gather}\label{firstWterm}\log x\sum_{R^2<s\leq \frac{x}{q^{1+\alpha}}}\tau(s)\sum_{\substack{t\leq \frac xs \\ t\equiv a\bar s\pmod q}}\tau(t)\ll \frac{x\log x}{q}\sum_{R<d\leq \sqrt{\frac x{q^{1+\alpha}}}}\frac{d^{2\alpha}}{d^2}\ll \frac{x\log x}{qR^{1-2\alpha}}.
\end{gather}
For the second term, we express the congruence condition via exponential sums.  We change notation from $t$ to $v$ to indicate that we have dropped the condition on $t$ being square-free, finding:
\begin{align*}\sum_{\substack{\frac{x}{q^{1+\alpha}}<s<q^{1+\alpha} \\ s=d^2 \\ (s,q)=1}}\sum_{\substack{v\leq \frac xs \\ v\equiv a\bar s\pmod q}}1=&\frac 1q\sum_{r=0}^{q-1}\sum_{\substack{\sqrt{\frac x{q^{1+\alpha}}}<d<q^{\frac 12+\frac \alpha 2}\\ (d,q)=1}}\sum_{\substack{r\leq \frac x{d^2} }}\exp\left(\frac{r(d^2v-a)}{q}\right)\\
\ll &\sqrt{\frac{x}{q^{1-\alpha}}}+\frac 1q\sum_{r=1}^{q-1}\sum_{\sqrt{\frac x{q^{1+\alpha}}}<d<q^{\frac 12+\frac \alpha 2}}\frac{1}{||d^2r/q||},
\end{align*}
where the first term is the $r=0$ term, and $||\cdot ||$ indicates the distance to the closest integer.  We change variables, letting $r'=d^2r$ and finding
\begin{align*}
\frac 1q\sum_{r=1}^{q-1}\sum_{\sqrt{\frac x{q^{1+\alpha}}}<d<q^{\frac 12+\frac \alpha 2}}\frac{1}{||d^2r/q||}=&\frac 1q\sum_{\sqrt{\frac x{q^{1+\alpha}}}<d<q^{\frac 12+\frac \alpha 2}}\sum_{r'=1}^{q-1}\frac{1}{||r'/q||}\\
\ll &\log x\sum_{\sqrt{\frac x{q^{1+\alpha}}}<d<q^{\frac 12+\frac \alpha 2}}1\\
\ll &q^{\frac 12+\frac 34\alpha}.
\end{align*}
Since $q\geq\sqrt x$,
$$\sqrt{\frac{x}{q^{1-\alpha}}}\leq q^{\frac 12+\frac 34\alpha},$$
and hence the $1\leq r\leq q-1$-term dominates the $r=0$-term, meaning that
\begin{align*}x^{o(1)}\sum_{\substack{\frac{x}{q^{1+\alpha}}<s<q^{1+\alpha} \\ s=d^2 \\ (s,q)=1}}&\sum_{\substack{t\leq \frac xs \\ t\equiv a\bar s\pmod q}}1\ll q^{\frac 12+\alpha}.
\end{align*}
Finally, with the third expression, note that the number of solutions to $d^2\equiv b\pmod q$ is $\ll 2^{\omega(q)}=x^{o(1)}$ if $d<q$, where $\omega(q)$ denotes the number of unique prime divisors of $q$. So
 \begin{align*}
&x^{o(1)}\sum_{\substack{R<t<\frac{x}{q^{1+\alpha}}  }}\sum_{\substack{q^{\frac 12+\frac \alpha 2}<d\leq \sqrt{\frac xt} \\ d^2\equiv a\bar t\pmod q}}1\ll x^{o(1)}\sum_{\substack{R<t<\frac{x}{q^{1+\alpha}}  }}1\ll \frac{x^{1+o(1)}}{q^{1+\alpha}}.
\end{align*}
This term is dominated by the term in (\ref{firstWterm}).

Putting all three of these results together then proves the first equation in the lemma.


For the second equation in the lemma, we write again $d^2=s$ and split the sum over $s$ into $s\leq x^\frac 34$ and $s>x^\frac 34$, noting that $\lambda'_R(s)\leq \tau(s)\log s\ll s^\alpha$:
\begin{align*}\left|S_W'(x,q)-S_R'(x,q)\right|\ll &\log x\sum_{R<d\leq x^\frac 38}\tau\left(d^2\right)\sum_{t\leq \frac x{d^2} }\tau(t)+x^{o(1)}\sum_{t<x^\frac 14}\sum_{d\leq \sqrt{\frac xt} }1\\
\ll &x\sum_{R<s\leq x^\frac 38}\frac{1}{d^{2-2\alpha}}+x^{\frac 12+o(1)}\sum_{t<x^\frac 14}\sqrt{\frac 1t}\\
\ll &\frac{x}{R^{1-2\alpha}}+x^{\frac 58+o(1)}.
\end{align*}
The first summand is clearly larger than the second, completing the proof of the lemma.
\end{proof}

\section{Prime Support}\label{final}

Next, we will show that the behavior of $S_W'(x,q)$ is largely the same as that of $\psi(x)$, while $S'_W(x,q,a)$ provides roughly the upper bound that one would expect for $\psi(x,q,a)$.  This will allow us to prove Theorem \ref{Main1}.


\begin{lemma}\label{SWx}
$$S_W'(x,q)=x+O\left(R\log^2 x+x\log xL(1,\chi)\right).$$
\end{lemma}
\begin{proof}
For a given $n$, write $n=n_1n_{-1}$, partitioned such that $p|n_j$ implies $\chi(p)=j$ for $\chi=\chi_D$.  (If no $p|n$ is such that $\chi(p)=j$ then we write $n_j=1$.)  Note that any prime $p|n$ must be such that $\chi(p)\neq 0$, since $n$ is $R$-rough and $D<R$.  So by Lemma \ref{split},
$$S_W'(x,q)=\sum_{n_1n_{-1}\leq x}\lambda'_W\left(n_1n_{-1}\right)=\sum_{n_1n_{-1}\leq x}\left[\lambda'_W\left(n_1\right)\lambda_W\left(n_{-1}\right)+\lambda_W\left(n_1\right) \lambda'_W\left(n_{-1}\right)\right].$$
Note that if $n_{-1}\neq 1$ then $\lambda_W\left(n_{-1}\right)=0$, since either $n_{-1}$ is square-free (and hence $\lambda\left(n_{-1}\right)=0$) or else $\mu(n_{-1})^2=0$.  Moreover, $\lambda'_W\left(n_{-1}\right)$ is exactly the same as $\Lambda\left(n_{-1}\right)$, since $\chi(d)=\mu(d)$ for any $d|n_{-1}$. So we break the sum into three cases: $n_{-1}=1$, $n_1=1$, and the rest.
\begin{gather}\label{DWLambda}
S_W'(x,q)=\sum_{R<n_1\leq x}\lambda'_W\left(n_1\right)+\sum_{\substack{R<p_{-1}\leq x }}\log\left(p_{-1}\right)+\sum_{\substack{n_1n_{-1}\leq x \\ n_1,n_{-1}>R }}\lambda_W\left(n_1\right) \lambda_W'\left(n_{-1}\right).
\end{gather}
For the first term, since $\lambda'_W\left(n_{1}\right)\leq \tau\left(n_1\right)\log x=\lambda\left(n_1\right)\log x$ for every $n_1$,
$$\sum_{R<n_1\leq x}\lambda'_W\left(n_1\right)\ll \log x\sum_{R<n_1\leq x}\lambda\left(n_1\right)\ll x\log xL(1,\chi).$$
For the middle term, we can apply the prime number theorem:
$$\sum_{\substack{R<p_{-1}\leq x }}\log\left(p_{-1}\right)=x-\sum_{\substack{R<m\leq x \\ p|m\Rightarrow \chi(p)=1}}\Lambda(m)+O\left(xe^{-(\log x)^{\frac 35-\varepsilon}}+x^\frac 12+\sum_{\substack{p^k\leq x\\ p|q\mbox{ or }p<R }}\Lambda(p)\right).$$
For the sum over $m$, we again have
$$\sum_{\substack{R<m\leq x \\ p|m\Rightarrow \chi(p)=1,p>R}}\Lambda(m)\leq \log x\sum_{R<m\leq x}\lambda(m)\ll x\log xL(1,\chi).$$
For the sum over $p|q$, we see that $k\leq \log x$ and the number of $p|q$ is bounded by $\log x$, and hence
$$\sum_{\substack{p^k\leq x \\ p|q,p>R}}\Lambda(p)\ll \log^3 x.$$
For the sum over $p<R$, we have
$$\sum_{\substack{p^k\leq x\\ p<R }}\Lambda(p)\ll R\log^2 x.$$
Since $R>xe^{-(\log x)^{\frac 35-\varepsilon}}$ by assumption, we then have
$$\sum_{\substack{n_{-1}\leq x \\ p|n_{-1}\Rightarrow p>R}}\Lambda\left(n_{-1}\right)=x+O\left(R\log^2 x+x\log xL(1,\chi)\right).$$
Finally, for the last term of (\ref{DWLambda}),
\begin{align*}
\sum_{\substack{n_1n_{-1}\leq x \\ n_1,n_{-1}>R}}\lambda_W\left(n_1\right) \Lambda\left(n_{-1}\right)
&\leq \sum_{R<n_{-1}<2\sqrt{x}}\Lambda\left(n_{-1}\right)\sum_{R<n_1\leq \frac x{n_{-1}} }\lambda\left(n_1\right) +\sum_{R<n_1\leq 2\sqrt x }\lambda\left(n_1\right)\sum_{R<n_{-1}<\frac x{n_{1}}}\Lambda\left(n_{-1}\right) \\
&\ll xL(1,\chi)\sum_{R<n_{-1}<\frac xR}\frac{\Lambda\left(n_{-1}\right)}{n_{-1}}+x\sum_{R<n_1\leq 2\sqrt x }\frac{\lambda\left(n_1\right)}{n_1}\\
&\ll xL(1,\chi)\log x.
\end{align*}
\end{proof}

\begin{lemma}\label{psibd}
$$\psi(x,q,a)\leq S_W'(x,q,a)+O\left(x^\frac 14\right).$$
\end{lemma}
\begin{proof}
Following the decomposition in (\ref{DWLambda}), we write
\begin{align*}
S_W'(x,q,a)\geq &\sum_{\substack{R<n_1\leq x \\ n_1\equiv a\pmod q}}\lambda'_W\left(n_1\right)+\sum_{\substack{R<p_{-1}\leq x\\ p_{-1}\equiv a\pmod q}}\log \left(p_{-1}\right)\\
=&\sum_{\substack{R<n_1\leq x \\ p|n_1\Rightarrow p>R \\ n_1\equiv a\pmod q}}\left(\lambda'_W\left(n_1\right)-\Lambda\left(n_1\right)\right)+\sum_{\substack{n\leq x \\ n\equiv a\pmod q}}\Lambda\left(n\right)\\
&+O\left(\sum_{k=2}^{\log x}\sum_{\substack{p\leq x^\frac 1k \\ p^k\equiv a\pmod q}}\log p\right)\\
\geq &\sum_{\substack{n\leq x \\ n\equiv a\pmod q}}\Lambda\left(n\right)+O\left(x^\frac 14+\sum_{\substack{p\leq \sqrt x \\ p^2\equiv a\pmod q}}\log p+\sum_{\substack{p\leq \sqrt x \\ p^3\equiv a\pmod q}}\log p\right),
\end{align*}
since $\lambda'_W(n_1)-\Lambda(n_1)\geq 0$ when $n_1$ is $R$-rough.  

By Chinese Remainder Theorem, the number of $k$-th power residues modulo $q$ is $$\frac{\phi(q)}{k^{\omega(q)C_k}}$$ where $C_k$ is bounded and depends on whether small powers of $k$ divide $q$, and hence for a given $a$, the number of possible $p$ for which $p^k\equiv a$ and $p<q$ is bounded by $$k^{\omega(q)C_k}< q^{\frac{C'_k}{\log\log q}}=x^{o(1)}.$$  Hence the last sum in the big-O term can be absorbed into $x^\frac 14$, yielding
\begin{align*}
S_W'(x,q,a)\geq &\sum_{\substack{n\leq x \\ n\equiv a\pmod q}}\Lambda\left(n\right)+O\left(x^\frac 14\right).
\end{align*}

\end{proof}

\section{Main Theorems: Theorem \ref{Main1}}
Finally, we can prove the first main theorem:
\begin{theorem}
Let $\sqrt x<q<D^{-1}x^{\frac 23-3\alpha}$ for any $\alpha>0$, let $(a,q)=1$, and let $A>0$.
Then
$$\psi(x,q,a)\leq \frac{1-\chi_D\left(\frac{aD}{(D,q)}\right)}{\phi(q)}\psi(x)+O(E(x)),$$
where
$$E(x)=\frac xqe^{-\frac{\varepsilon\log x}{24\log R}
}+\frac{x\log^5 xL(1,\chi)}{\phi(q)}+\frac{R\log^2 x}{\phi(q)}.$$
Moreover, if $q$ satisfies the bounds given above and $h=h(x)$ is a function with $h<1$, then the equation
$$\psi(x,q,a)\geq \frac{1-h+\chi_D\left(\frac{aD}{(D,q)}\right)}{\phi(q)}\psi(x),$$
holds for all but $$O\left(\frac{\phi(q)^2E(x)}{hx}\right)$$values of $a$ with $(a,q)=1$.
\end{theorem}
\begin{proof}
From Lemma \ref{DW},
\begin{align*}
\psi(x,q,a)\leq S_W'(x,q,a)+O\left(x^\frac 14\right).
\end{align*}

From Lemmas \ref{DR}, \ref{SWx}, and \ref{psibd},
\begin{align*}
S_W'(x,q,a)=&S_R'(x,q,a)+O\left(q^{\frac 12+\alpha}+\frac{x\log x}{qR^{1-2\alpha}}\right)\\
=&\frac{1-\chi_D\left(\frac{aD}{(D,q)}\right)}{\phi(q)}S_R'(x,q)+O\left((Dq)^{\frac 12+\alpha}+\frac xqe^{-\frac{\alpha\log x}{24\log R}}+\frac{x\log^5 xL(1,\chi)}{\phi(q)}\right)\\
=&\frac{1-\chi_D\left(\frac{aD}{(D,q)}\right)}{\phi(q)}S_W'(x,q)+O\left(\frac xqe^{-\frac{\alpha\log x}{24\log R}}+\frac{x\log^5 xL(1,\chi)}{\phi(q)}\right)\\
=&\frac{x\left(1-\chi_D\left(\frac{aD}{(D,q)}\right)\right)}{\phi(q)}+O\left(\frac xqe^{-\frac{\alpha\log x}{24\log R}}+\frac{x\log^5 xL(1,\chi)}{\phi(q)}+\frac{R\log^2 x}{\phi(q)}\right),
\end{align*}
where we have absorbed smaller terms into larger ones.  Letting $E(x)$ be the error term then gives us
\begin{gather}\label{psibd1} \psi(x,q,a)\leq \frac{x\left(1-\chi_D\left(\frac{aD}{(D,q)}\right)\right)}{\phi(q)}+E(x),\end{gather}
proving the first half of the theorem.

For the second half of the theorem, we have trivially
\begin{gather}\label{PNT}
\sum_{a\in \mathbb Z_q^*}\psi(x,q,a)=\psi(x)=x+O\left(xe^{-(\log x)^{\frac 35-\varepsilon}}\right),
\end{gather}
and
$$\sum_{a\in \mathbb Z_q^*}\frac{\left(1-\chi_D\left(\frac{aD}{(D,q)}\right)\right)}{\phi(q)}=1.$$
So let $\mathcal C$ be the set of $a$ such that
$$\psi(x,q,a)\leq \frac{(1-h)x}{\phi(q)}$$
for some $h$.  Bounding the error term in (\ref{PNT}) with $R$,
\begin{align*}x+O\left(R\right)=&\sum_{a\in \mathbb Z_q^*}\psi(x,q,a)\\
\leq &(\phi(q)-|\mathcal C|)\left(\frac{x}{\phi(q)}+O\left(E(x)\right)\right)+|\mathcal C|\frac{(1-h)x}{\phi(q)}\\
=& x-\frac{hx|\mathcal C|}{\phi(q)}+O\left((\phi(q)-|\mathcal C|)E(x)\right).
\end{align*}
Collapsing down the inequalities, we have
$$x+O\left(R\right)\leq x-\frac{hx|\mathcal C|}{\phi(q)}+O\left((\phi(q)-|\mathcal C|)E(x)\right).$$
Since $R\ll E(x)$, this can be rewritten as
$$|\mathcal C|\ll \frac{\phi(q)}{hx}\left(\phi(q)E(x)\right).$$
This completes the theorem.
\end{proof}
Theorem \ref{Main1.5} gives a stronger - and simpler - bound if $D|q$ and $\chi(a)=1$.  We prove this with the following theorem.

\begin{theorem}
Let $\sqrt x<q<D^{-1}x^{\frac 23-3\alpha}$ for any $\alpha>0$, let $(a,q)=1$, let $A>0$, and let $D|q$.  If $\chi(a)=1$ then
$$\psi(x,q,a)\leq \frac{xL(1,\chi)}{q}+O\left((Dq)^{\frac 12+\alpha}\right).$$
\end{theorem}
\begin{proof}
If $p$ is prime and $\chi(p)=1$ then $\lambda(p)=2$ and $\lambda(p^k)=k+1$.  So if $\chi(a)=1$ then
\begin{align*}&\sum_{\substack{n\leq x \\ n\equiv a\pmod q}}2\Lambda(n)\\
&\leq \log x\sum_{\substack{n\leq x \\ n\equiv a\pmod q}}\lambda(n)+\log x\sum_{\substack{p\leq \sqrt x \\ p^2\equiv a\pmod q}}1+\log x\sum_{\substack{p\leq \sqrt x \\ p^3\equiv a\pmod q}}1+O\left(\log x\sum_{k=4}^{\log x}\sum_{\substack{p\leq x^\frac 1k \\ p^k\equiv a\pmod q}}1\right)\\
&=\log x\sum_{\substack{n\leq x \\ n\equiv a\pmod q}}\lambda(n)+O\left(x^\frac 14\log x\right)\\
&=(\log x)S(x,q,a)+O\left(x^\frac 14\log x\right)\\
&=(\log x)S(x,q,a)+O\left(x^\frac 14\log x+(Dq)^{\frac 12+\alpha}\right)
\end{align*}
by Lemma \ref{lambdaap}, where the bound on the error terms from the second to third lines is as in the proof of Lemma \ref{psibd}.  Since
$$S(x,q)=\frac{x\phi(q)L(1,\chi)}{q}+O\left(\sqrt{Dx}\right)$$
by \cite[Lemma 5.1]{FI03}, and since $$x^\frac 14\log x\ll (Dq)^{\frac 12+\alpha},$$
the theorem then follows.
\end{proof}

\section{Main Theorems: Theorem \ref{Main2}}\label{final2}

Finally, we prove the following, from which one easily deduces Theorem \ref{Main2}:

\begin{theorem}
Let $\sqrt x<Q<D^{-1}x^{\frac 23-\alpha}$ for any $\alpha>0$.  Then for any fixed integer $a\neq 0$,
$$\sum_{\substack{q\sim Q \\ (a,q)=1}}\left|\psi(x,q,a)-\frac{\psi(x)}{\phi(q)}\right|\ll xe^{-\frac{\alpha\log x}{24\log R}
}+x\log^5 xL(1,\chi)+R\log^2 x.$$
\end{theorem}

\begin{proof}
We again recall the decomposition in (\ref{DWLambda}):
\begin{gather}
S_W'(x,q,a)=\sum_{\substack{R<n_1\leq x \\ n_1\equiv a\pmod q}}\lambda'_W\left(n_1\right)+\sum_{\substack{R<p_{-1}\leq x \\ p_{-1}\equiv a\pmod q}}\log\left(p_{-1}\right)+\mathop{\sum\sum}\limits_{\substack{n_1n_{-1}\leq x \\ n_1,n_{-1}>R \\ n_1n_{-1}\equiv a\pmod q}}\lambda_W\left(n_1\right) \lambda_W'\left(n_{-1}\right).
\end{gather}
We show first that when summed over $q\sim Q$, $S_W'(x,q,a)$ will generally be a good approximation for $\psi(x,q,a)$.  To this end,
\begin{align*}
\sum_{\substack{q\sim Q  \\ (a,q)=1}}&\left|S_W'(x,q,a)-\psi(x,q,a)\right|\\
\leq &\sum_{\substack{q\sim Q  \\ (a,q)=1}}\left(\sum_{\substack{R<n_1\leq x \\ n_1\equiv a\pmod q}}\left(\lambda'_W\left(n_1\right)-\Lambda(n_1)\right)+\mathop{\sum\sum}\limits_{\substack{n_1n_{-1}\leq x \\ n_1,n_{-1}>R \\ n_1n_{-1}\equiv a\pmod q}}\lambda_W\left(n_1\right) \lambda_W'\left(n_{-1}\right)+O\left(x^\frac 14\right)\right)
\end{align*}
Recall that $\lambda'_W\left(n_1\right)-\Lambda(n_1)\geq \lambda(n_1)\log x$.  So by Lemma \ref{lambdaap},
\begin{align*}
\sum_{\substack{q\sim Q  \\ (a,q)=1}}\sum_{\substack{R<n_1\leq x \\ n_1\equiv a\pmod q}}\left(\lambda'_W\left(n_1\right)-\Lambda(n_1)\right)\ll &\log x\sum_{\substack{q\sim Q  \\ (a,q)=1}}\sum_{\substack{R<n_1\leq x \\ n_1\equiv a\pmod q}}\lambda\left(n_1\right)\\
\ll &\log x\sum_{\substack{q\sim Q  \\ (a,q)=1}}\frac{1}{\phi(q)}\sum_{\substack{R<n_1\leq x }}\lambda\left(n_1\right)\\
\ll &xL(1,\chi)\log x.
\end{align*}
Hence
\begin{align*}
\sum_{\substack{q\sim Q  \\ (a,q)=1}}&\left|S_W'(x,q,a)-\psi(x,q,a)\right|\ll &\log x\sum_{\substack{q\sim Q  \\ (a,q)=1}}\left(\mathop{\sum\sum}\limits_{\substack{n_1n_{-1}\leq x \\ n_1,n_{-1}>R \\ n_1n_{-1}\equiv a\pmod q}}\lambda_W\left(n_1\right) \right)+O\left(xL(1,\chi)\log x+Qx^\frac 14\right).
\end{align*}
Now, if $n_1n_{-1}\equiv a \pmod q$ then $q|n_1n_{-1}-a$, and hence we can write
\begin{align*}
\sum_{\substack{q\sim Q  \\ (a,q)=1}}\mathop{\sum\sum}\limits_{\substack{n_1n_{-1}\leq x \\ n_1,n_{-1}>R \\ n_1n_{-1}\equiv a\pmod q}}\lambda_W\left(n_1\right) \leq \mathop{\sum\sum}\limits_{\substack{n_1n_{-1}\leq x \\ n_1,n_{-1}>R }}\lambda_W\left(n_1\right) \tau(n_1n_{-1}-a).
\end{align*}
By Lemma \ref{Mun}, we can bound this with
\begin{align*}
\ll &\sum_{d\leq x^{\frac 13-\alpha}} \tau(d)\mathop{\sum\sum}\limits_{\substack{n_1n_{-1}\leq x \\ n_1,n_{-1}>R \\ n_1n_{-1}\equiv a\pmod d}}\lambda_W\left(n_1\right) .
\end{align*}
Note that if $n_{-1}\geq x^\frac 14>d^{1+\alpha}$ then
$$\sum_{\substack{x^{\frac 13}\leq n_{-1}\leq \frac x{n_1} \\ n_{-1}\equiv a\overline{n_1}\pmod d}}1\ll \frac 1d\sum_{\substack{n_{-1}\leq \frac x{n_1} }}1,$$
while if $n_{-1}<x^\frac 14$ then $n_{1}\geq x^\frac 23$, and hence
$$\sum_{\substack{x^{\frac 23}\leq n_1\leq \frac x{n_{-1}} \\ n_{1}\equiv a\overline{n_{-1}}\pmod d}}\lambda_W(n_1)\ll \sum_{\substack{x^{\frac 23}\leq n_1\leq \frac x{n_{-1}} \\ n_{1}\equiv a\overline{n_{-1}}\pmod d}}\lambda(n_1)\ll  \frac 1{\phi(d)}\sum_{\substack{n_1\leq \frac x{n_{-1}} }}\lambda(n_1).$$
So
\begin{align*}
\sum_{d\leq x^{\frac 13-\alpha}} &\tau(d)\mathop{\sum\sum}\limits_{\substack{n_1n_{-1}\leq x \\ n_1,n_{-1}>R \\ n_1n_{-1}\equiv a\pmod d}}\lambda_W\left(n_1\right) \\
\ll &\sum_{d\leq x^{\frac 13-\alpha}} \frac{\tau(d)}{\phi(d)}\mathop{\sum\sum}\limits_{\substack{n_1n_{-1}\leq x \\ n_1,n_{-1}>R }}\lambda\left(n_1\right) \\
\ll &x\sum_{d\leq x^{\frac 13-\alpha}} \frac{\tau(d)}{d}\left[L(1,\chi)\mathop{\sum}\limits_{\substack{ R<n_{-1}\leq 2x^\frac 23}}\frac{1}{n_{-1}}+\mathop{\sum}\limits_{\substack{R<n_{-1}\leq 2x^\frac 14}}\frac{\lambda(n_1)}{n_{1}}\right]\\
\ll &xL(1,\chi)\log^{3} x.
\end{align*}
Hence
\begin{align*}
\sum_{\substack{q\sim Q  \\ (a,q)=1}}&\left|S_W'(x,q,a)-\psi(x,q,a)\right|\ll xL(1,\chi)\log^4 x+Qx^\frac 14\ll xL(1,\chi)\log^4 x.
\end{align*}
From Lemma \ref{DW}, if $Q<D^{-1}x^{\frac 23-3\alpha}$ then
$$\sum_{\substack{q\sim Q  \\ (a,q)=1}}\left|S_W'(x,q,a)-\frac{1}{\phi(q)}S_W'(x,q)\right|\ll xe^{-\frac{\alpha\log x}{24\log R}}+x\log^5 xL(1,\chi),$$
and we recall from Lemma \ref{SWx} that
$$S'_W(x,q)=\psi(x)+O\left(R\log^2 x+x\log xL(1,\chi)\right).$$
Choosing the largest terms from these expressions, the theorem then follows.
\end{proof}

\bibliographystyle{line}

\end{document}